\documentclass{amsart}
\usepackage{amstext,amssymb,amsthm,amsopn,newlfont,graphpap,graphics,graphicx,mathrsfs}
\allowdisplaybreaks
\usepackage[parfill]{parskip}
\usepackage[noadjust]{cite}
\usepackage{epigraph}
\usepackage[colorlinks=true,
            linkcolor=red,
            urlcolor=blue,
            citecolor=magenta]{hyperref}
\usepackage{color}
\usepackage{mathrsfs}
\allowdisplaybreaks
\theoremstyle{plain}
\newtheorem{thm}{Theorem}[section]
\newtheorem*{thm*}{Theorem}
\newtheorem{prop}{Proposition}[section]
\newtheorem*{prop*}{Proposition}
\newtheorem{cor}{Corollary}[section]
\newtheorem*{cor*}{Corollary}

\newtheorem*{lem*}{Lemma}
\theoremstyle{definition}
\newtheorem{defn}{Definition}[section]
\newtheorem*{defn*}{Definition}

\newtheorem*{exmp*}{Example}

\newtheorem*{exmps*}{Examples}

\newtheorem*{rem*}{Remark}

\newtheorem*{rems*}{Remarks}

\newtheorem*{note*}{Note}
\newcommand{\N}{{\mathbb N}}
\newcommand{\Z}{{\mathbb Z}}
\newcommand{\C}{{\mathbb C}}

\DeclareMathOperator{\Rep}{Re\,}
\DeclareMathOperator{\Imp}{Im\,}
\DeclareMathOperator{\dist}{dist}
\DeclareMathOperator{\spa}{span}
\begin{document}
\title[On the Gevrey ultradifferentiability of weak solutions]
{On the Gevrey ultradifferentiability\\ 
of weak solutions\\
of an abstract evolution equation\\
with a scalar type spectral operator}
\author[Marat V. Markin]{Marat V. Markin}
\address{
Department of Mathematics\newline
California State University, Fresno\newline
5245 N. Backer Avenue, M/S PB 108\newline
Fresno, CA 93740-8001, USA
}
\email{mmarkin@csufresno.edu}
\dedicatory{To Dr. Valentina I. Gorbachuk, a remarkable person and mathematician, in honor of her jubilee.}
\begin{abstract}
Found are conditions on a scalar type spectral operator $A$ in a complex Banach space necessary and sufficient
for all weak solutions of the evolution equation
\begin{equation*}
y'(t)=Ay(t),\ t\ge 0,
\end{equation*}
to be strongly Gevrey ultradifferentiable of order $\beta\ge 1$, in particular analytic or entire, on $[0,\infty)$. Certain inherent smoothness improvement effects are analyzed.
\end{abstract}
\maketitle
\epigraph{\textit{It is impossible to be a mathematician without being a poet in soul.}}{Sofia Kovalevskaya}
\section[Introduction]{Introduction}

We find conditions on a scalar type spectral operator $A$ in a complex Banach space $X$ necessary and sufficient for all \textit{weak solutions} of the evolution equation
\begin{equation}\label{1}
y'(t)=Ay(t),\ t\ge 0,
\end{equation}
to be strongly \textit{Gevrey ulradifferentiable} of order $\beta\ge 1$, in particular \textit{analytic} or \textit{entire}, on $[0,\infty)$ and analyze certain inherent smoothness improvement effects to generalize the corresponding results for equation \eqref{1} with a \textit{normal operator} $A$ in a complex Hilbert space \cite{Markin2001(1)}.

The results of the present paper develop those of \cite{Markin2011}, where similar consideration is given to the strong differentiability of the weak solutions of \eqref{1} on $[0,\infty)$ and $(0,\infty)$.

\begin{defn}[Weak Solution]\label{ws}\ \\
Let $A$ be a closed densely defined linear operator in a Banach space $X$. A strongly continuous vector function $y:[0,\infty)\rightarrow X$ is called a {\it weak solution} of equation \eqref{1} if, for any $g^* \in D(A^*)$,
\begin{equation*}
\dfrac{d}{dt}\langle y(t),g^*\rangle = \langle y(t),A^*g^* \rangle,\ t\ge 0,
\end{equation*}
where $D(\cdot)$ is the \textit{domain} of an operator, $A^*$ is the operator {\it adjoint} to $A$, and $\langle\cdot,\cdot\rangle$ is the {\it pairing} between
the space $X$ and its dual $X^*$ (see \cite{Ball}).
\end{defn}

Due to the \textit{closedness} of $A$, the weak solution of \eqref{1} can be equivalently defined to be a strongly continuous vector function $y:[0,\infty)\mapsto X$ such that, for all $t\ge 0$,
\begin{equation*}
\int_0^ty(s)\,ds\in D(A)\ \text{and} \ y(t)=y(0)+A\int_0^ty(s)\,ds
\end{equation*}
and is also called a \textit{mild solution} (cf. {\cite[Ch. II, Definition 6.3]{Engel-Nagel}}, see also {\cite[Preliminaries]{Markin2017(2)}}).

Such a notion of \textit{weak solution}, which need not be differentiable in the strong sense, generalizes that of \textit{classical} one, strongly differentiable on $[0,\infty)$ and satisfying the equation in the traditional plug-in sense, the classical solutions being precisely the weak ones strongly differentiable on $[0,\infty)$.

When a closed densely defined linear operator $A$
in a complex Banach space $X$ generates a $C_0$-semigroup $\left\{T(t) \right\}_{t\ge 0}$ of  bounded linear operators (see, e.g., \cite{Hille-Phillips,Engel-Nagel}), i.e., the associated \textit{abstract Cauchy problem} (\textit{ACP})
\begin{equation}\label{ACP}
\begin{cases}
y'(t)=Ay(t),\ t\ge 0,\\
y(0)=f
\end{cases}
\end{equation}
is \textit{well-posed} (cf. {\cite[Ch. II, Definition 6.8]{Engel-Nagel}}), the weak solutions of equation \eqref{1} are the orbits
\begin{equation}\label{semigroup}
y(t)=T(t)f,\ t\ge 0,
\end{equation}
with $f\in X$ {\cite[Ch. II, Proposition 6.4]{Engel-Nagel}} (see also {\cite[Theorem]{Ball}}), whereas the classical ones are those with $f\in D(A)$
(see, e.g., {\cite[Ch. II, Proposition 6.3]{Engel-Nagel}}). 

Observe that, in our discourse, the associated \textit{ACP} may be \textit{ill-posed}, i.e., the scalar type spectral operator $A$ need not generate a $C_0$-semigroup (cf. \cite{Markin2002(2)}). 

\section[Preliminaries]{Preliminaries}

For the reader's convenience, we outline in this section certain essential preliminaries.

\subsection{Scalar Type Spectral Operators}\ 

Henceforth, unless specified otherwise, $A$ is supposed to be a {\it scalar type spectral operator} in a complex Banach space $(X,\|\cdot\|)$ and $E_A(\cdot)$ to be its
strongly $\sigma$-additive \textit{spectral measure} (the \textit{resolution of the identity}) assigning to each {\it Borel set} $\delta$ of the \textit{complex plane} $\C$ a {\it projection operator} $E_A(\delta)$ on $X$ and having the operator's \textit{spectrum} $\sigma(A)$ as its {\it support} \cite{Survey58,Dun-SchIII}.

Observe that, in a complex finite-dimensional space, 
the scalar type spectral operators are those linear operators on the space, for which there is an \textit{eigenbasis} (see, e.g., \cite{Survey58,Dun-SchIII}) and, in a complex Hilbert space, the scalar type spectral operators are precisely those that are similar to the {\it normal} ones \cite{Wermer}.

Associated with a scalar type spectral operator in a complex Banach space is the {\it Borel operational calculus} analogous to that for a \textit{normal operator} in a complex Hilbert space \cite{Survey58,Dun-SchII,Dun-SchIII,Plesner}, which assigns to any Borel measurable function $F:\sigma(A)\to \C$ a scalar type spectral operator
\begin{equation*}
F(A):=\int\limits_{\sigma(A)} F(\lambda)\,dE_A(\lambda)
\end{equation*}
defined as follows:
\[
F(A)f:=\lim_{n\to\infty}F_n(A)f,\ f\in D(F(A)),\
D(F(A)):=\left\{f\in X\middle| \lim_{n\to\infty}F_n(A)f\ \text{exists}\right\},
\]
where
\begin{equation*}
F_n(\cdot):=F(\cdot)\chi_{\{\lambda\in\sigma(A)\,|\,|F(\lambda)|\le n\}}(\cdot),
\ n\in\N,
\end{equation*}
($\chi_\delta(\cdot)$ is the {\it characteristic function} of a set $\delta\subseteq \C$, $\N:=\left\{1,2,3,\dots\right\}$ is the set of \textit{natural numbers}) and
\begin{equation*}
F_n(A):=\int\limits_{\sigma(A)} F_n(\lambda)\,dE_A(\lambda),\ n\in\N,
\end{equation*}
are {\it bounded} scalar type spectral operators on $X$ defined in the same manner as for a {\it normal operator} (see, e.g., \cite{Dun-SchII,Plesner}).

In particular,
\begin{equation}\label{power}
A^n=\int\limits_{\sigma(A)} \lambda^n\,dE_A(\lambda),\ n\in\Z_+,
\end{equation}
($\Z_+:=\left\{0,1,2,\dots\right\}$ is the set of \textit{nonnegative integers}, $A^0:=I$, $I$ is the \textit{identity operator} on $X$) and
\begin{equation}\label{exp}
e^{zA}:=\int\limits_{\sigma(A)} e^{z\lambda}\,dE_A(\lambda),\ z\in\C.
\end{equation}

The properties of the {\it spectral measure} and {\it operational calculus}, exhaustively delineated in \cite{Survey58,Dun-SchIII}, underlie the entire subsequent discourse. Here, we touch upon a few facts of special importance.

Due to its {\it strong countable additivity}, the spectral measure $E_A(\cdot)$ is {\it bounded} \cite{Dun-SchI,Dun-SchIII}, i.e., there is such an $M>0$ that, for any Borel set $\delta\subseteq \C$,
\begin{equation}\label{bounded}
\|E_A(\delta)\|\le M.
\end{equation}
Observe that the notation $\|\cdot\|$ is recycled here to designate the norm in the space $L(X)$ of all bounded linear operators on $X$. We shall adhere to this rather conventional economy of symbols in what follows adopting the same notation for the norm in the dual space $X^*$ as well (cf. \cite{Engel-Nagel,Markin2002(2)}). 

For any $f\in X$ and $g^*\in X^*$, the \textit{total variation} $v(f,g^*,\cdot)$ of the complex-valued Borel measure $\langle E_A(\cdot)f,g^* \rangle$ is a {\it finite} positive Borel measure with
\begin{equation}\label{tv}
v(f,g^*,\C)=v(f,g^*,\sigma(A))\le 4M\|f\|\|g^*\|
\end{equation}
(see, e.g., \cite{Markin2004(1),Markin2004(2)}).

Also (Ibid.), for a Borel measurable function $F:\C\to \C$, $f\in D(F(A))$, $g^*\in X^*$, and a Borel set $\delta\subseteq \C$,
\begin{equation}\label{cond(ii)}
\int\limits_\delta|F(\lambda)|\,dv(f,g^*,\lambda)
\le 4M\|E_A(\delta)F(A)f\|\|g^*\|.
\end{equation}
In particular, for $\delta=\sigma(A)$,
\begin{equation}\label{cond(i)}
\int\limits_{\sigma(A)}|F(\lambda)|\,d v(f,g^*,\lambda)\le 4M\|F(A)f\|\|g^*\|.
\end{equation}

Observe that the constant $M>0$ in \eqref{tv}--\eqref{cond(i)} is from 
\eqref{bounded}.

Further, for a Borel measurable function $F:\C\to [0,\infty)$, a Borel set $\delta\subseteq \C$, a sequence $\left\{\Delta_n\right\}_{n=1}^\infty$ 
of pairwise disjoint Borel sets in $\C$, and 
$f\in X$, $g^*\in X^*$,
\begin{equation}\label{decompose}
\int\limits_{\delta}F(\lambda)\,dv(E_A(\cup_{n=1}^\infty \Delta_n)f,g^*,\lambda)
=\sum_{n=1}^\infty \int\limits_{\delta\cap\Delta_n}F(\lambda)\,dv(E_A(\Delta_n)f,g^*,\lambda).
\end{equation}

Indeed, since, for any Borel sets $\delta,\sigma\subseteq \C$,
\begin{equation*}
E_A(\delta)E_A(\sigma)=E_A(\delta\cap\sigma)
\end{equation*}
\cite{Survey58,Dun-SchIII}, 
for the total variation,
\begin{equation*}
v(E_A(\delta)f,g^*,\sigma)=v(f,g^*,\delta\cap\sigma).
\end{equation*}

Whence, due to the {\it nonnegativity} of $F(\cdot)$ (see, e.g., \cite{Halmos}),
\begin{multline*}
\int\limits_\delta F(\lambda)\,dv(E_A(\cup_{n=1}^\infty \Delta_n)f,g^*,\lambda)
=\int\limits_{\delta\cap\cup_{n=1}^\infty \Delta_n}F(\lambda)\,dv(f,g^*,\lambda)
\\
\ \
=\sum_{n=1}^\infty \int\limits_{\delta\cap\Delta_n}F(\lambda)\,dv(f,g^*,\lambda)
=\sum_{n=1}^\infty \int\limits_{\delta\cap\Delta_n}F(\lambda)\,dv(E_A(\Delta_n)f,g^*,\lambda).
\hfill
\end{multline*}

The following statement, allowing to characterize the domains of Borel measurable functions of a scalar type spectral operator in terms of positive Borel measures, is fundamental for our discourse.

\begin{prop}[{\cite[Proposition $3.1$]{Markin2002(1)}}]\label{prop}\ \\
Let $A$ be a scalar type spectral operator in a complex Banach space $(X,\|\cdot\|)$ with spectral measure $E_A(\cdot)$ and $F:\sigma(A)\to \C$ be a Borel measurable function. Then $f\in D(F(A))$ iff
\begin{enumerate}
\item[(i)] for each $g^*\in X^*$, 
$\displaystyle \int\limits_{\sigma(A)} |F(\lambda)|\,d v(f,g^*,\lambda)<\infty$ and
\item[(ii)] $\displaystyle \sup_{\{g^*\in X^*\,|\,\|g^*\|=1\}}
\int\limits_{\{\lambda\in\sigma(A)\,|\,|F(\lambda)|>n\}}
|F(\lambda)|\,dv(f,g^*,\lambda)\to 0,\ n\to\infty$,
\end{enumerate}
where $v(f,g^*,\cdot)$ is the total variation of $\langle E_A(\cdot)f,g^* \rangle$.
\end{prop} 

The succeeding key theorem provides a full description of the weak solutions of equation \eqref{1} with a scalar type spectral operator $A$ in a complex Banach space.

\begin{thm}[{\cite[Theorem $4.2$]{Markin2002(1)}}]\label{GWS}\ \\
Let $A$ be a scalar type spectral operator in a complex Banach space $(X,\|\cdot\|)$. A vector function $y:[0,\infty) \to X$ is a weak solution 
of equation \eqref{1} iff there is an $\displaystyle f \in \bigcap_{t\ge 0}D(e^{tA})$ such that
\begin{equation}\label{expf}
y(t)=e^{tA}f,\ t\ge 0,
\end{equation}
the operator exponentials understood in the sense of the Borel operational calculus (see \eqref{exp}).
\end{thm}

Theorem \ref{GWS} generalizes {\cite[Theorem $3.1$]{Markin1999}}, its counterpart for a normal operator $A$ in a complex Hilbert space, and implies, in particular,
\begin{itemize}
\item that the subspace $\bigcap_{t\ge 0}D(e^{tA})$ of all possible initial values of the weak solutions of equation \eqref{1} is the largest permissible for the exponential form given by \eqref{expf}, which highlights the contextual naturalness of the notion of weak solution, and
\item that associated \textit{ACP} \eqref{ACP}, whenever solvable,  is solvable \textit{uniquely}.
\end{itemize}

Observe that the initial-value subspace $\bigcap_{t\ge 0}D(e^{tA})$ of equation \eqref{1}, containing the dense in $X$ subspace $\bigcup_{\alpha>0}E_A(\Delta_\alpha)X$, where
\begin{equation*}
\Delta_\alpha:=\left\{\lambda\in\C\,\middle|\,|\lambda|\le \alpha \right\},\ \alpha>0,
\end{equation*}
which coincides with the class ${\mathscr E}^{\{0\}}(A)$ of \textit{entire} vectors of $A$ of \textit{exponential type} (see below), is \textit{dense} in $X$ as well.

When a scalar type spectral operator $A$ in a complex Banach space generates a $C_0$-semigroup $\left\{T(t) \right\}_{t\ge 0}$, 
\[
T(t)=e^{tA}\ \text{and}\ D(e^{tA})=X,\ t\ge 0,
\]
\cite{Markin2002(2)}, and hence, Theorem \ref{GWS} is consistent with the well-known description of the weak solutions for this setup (see \eqref{semigroup}).

We also need the following characterization of a particular weak solution's of equation \eqref{1} with a scalar type spectral operator $A$ in a complex Banach space being strongly infinite differentiable on a subinterval $I$ of $[0,\infty)$.

\begin{prop}[{\cite[Corollary $3.2$]{Markin2011}} with $T=\infty$]\label{Cor}\ \\
Let $A$ be a scalar type spectral operator in a complex Banach space $(X,\|\cdot\|)$ and $I$ be a subinterval of $[0,\infty)$.
A weak solution $y(\cdot)$ of equation \eqref{1} is strongly infinite differentiable on $I$ iff, for each $t\in I$, 
\begin{equation*}
y(t) \in C^\infty(A),
\end{equation*}
in which case
\begin{equation*}
y^{(n)}(t)=A^ny(t),\ n\in \N,t\in I.
\end{equation*}
\end{prop}

Subsequently, the frequent terms {\it ``spectral measure"} and {\it ``operational calculus"} are abbreviated to {\it s.m.} and {\it o.c.}, respectively.

\subsection{Gevrey Classes of Functions}\label{GCF}\

\begin{defn}[Gevrey Classes of Functions]\ \\
Let $(X,\|\cdot\|)$ be a (real or complex) Banach space, $C^\infty(I,X)$ be the space of all $X$-valued functions strongly infinite differentiable on an interval $I\subseteq (-\infty,\infty)$, and $0\le \beta<\infty$.

The following subspaces of $C^\infty(I,X)$
\begin{align*}
{\mathscr E}^{\{\beta\}}(I,X):=\bigl\{g(\cdot)\in C^{\infty}(I, X) \bigm |&
\forall [a,b] \subseteq I\ \exists \alpha>0\ \exists c>0:
\\
&\max_{a \le t \le b}\|g^{(n)}(t)\| \le c\alpha^n [n!]^\beta,
\ n\in\Z_+\bigr\},\\
{\mathscr E}^{(\beta)}(I,X):= \bigl\{g(\cdot) \in C^{\infty}(I,X) \bigm |& 
\forall [a,b] \subseteq I\ \forall \alpha > 0 \ \exists c>0:
\\
&\max_{a \le t \le b}\|g^{(n)}(t)\| \le c\alpha^n [n!]^\beta,
\ n\in\Z_+\bigr\},
\end{align*}
are called the {\it $\beta$th-order Gevrey classes} of strongly ultradifferentiable vector functions on $I$ of {\it Roumieu} and {\it Beurling type}, respectively (see, e.g., \cite{Gevrey,Komatsu1,Komatsu2,Komatsu3}).
\end{defn}

In view of {\it Stirling's formula}, the 
sequence $\left\{ [n!]^\beta\right\}_{n=0}^\infty$ can be replaced with
$\left\{ n^{\beta n}\right\}_{n=0}^\infty$.

For $0\le\beta<\beta'<\infty$, the inclusions
\begin{equation*}
{\mathscr E}^{(\beta)}(I,X)\subseteq{\mathscr E}^{\{\beta\}}(I,X)
\subseteq {\mathscr E}^{(\beta')}(I,X)\subseteq
{\mathscr E}^{\{\beta'\}}(I,X)\subseteq C^{\infty}(I,X)
\end{equation*}
hold.

\begin{itemize}
\item For $1<\beta<\infty$, the Gevrey classes
are \textit{non-quasianalytic} (see, e.g., \cite{Komatsu2}).
\item For $\beta=1$, ${\mathscr E}^{\{1\}}(I,X)$ 
is the class of all {\it analytic} on $I$, i.e., {\it analytically continuable} into complex neighborhoods of $I$, vector functions and ${\mathscr E}^{(1)}(I,X)$ is the class of all {\it entire}, i.e., allowing {\it entire} continuations, vector functions \cite{Mandel}.
\item For $0\le\beta<1$, the Gevrey class ${\mathscr E}^{\{\beta\}}(I,X)$ (${\mathscr E}^{(\beta)}(I,X)$) consists of all functions $g(\cdot)\in {\mathscr E}^{(1)}(I,X)$ such that, for some (any) $\gamma>0$, there is an $M>0$ for which
\begin{equation}\label{order}
\|g(z)\|\le Me^{\gamma|z|^{1/(1-\beta)}},\ z\in \C,
\end{equation}
\cite{Markin2001(2)}. In particular,
for $\beta=0$, ${\mathscr E}^{\{0\}}(I,X)$ and ${\mathscr E}^{(0)}(I,X)$ are the classes of entire vector functions of \textit{exponential} and \textit{minimal exponential type}, respectively (see, e.g., \cite{Levin}).
\end{itemize} 

\subsection{Gevrey Classes of Vectors}\

One can consider the Gevrey classes in a more general sense. 

\begin{defn}[Gevrey Classes of Vectors]\ \\
Let $(A,D(A))$ be a densely defined closed linear operator in a (real or complex) Banach space $(X,\|\cdot\|)$, $0\le \beta<\infty$, and
\begin{equation*}
C^{\infty}(A):=\bigcap_{n=0}^{\infty}D(A^n)
\end{equation*}
be the subspace of infinite differentiable vectors of $A$.

The following subspaces of $C^{\infty}(A)$
\begin{align*}
{\mathscr E}^{\{\beta\}}(A)&:=\left\{x\in C^{\infty}(A)\, \middle |\, 
\exists \alpha>0\ \exists c>0:
\|A^nx\| \le c\alpha^n [n!]^\beta,\ n\in\Z_+ \right\},\\
{\mathscr E}^{(\beta)}(A)&:=\left\{x \in C^{\infty}(A)\, \middle|\,\forall \alpha > 0 \ \exists c>0:
\|A^nx\| \le c\alpha^n [n!]^\beta,\ n\in\Z_+ \right\}
\end{align*}
are called the \textit{$\beta$th-order Gevrey classes} of ultradifferentiable vectors of $A$ of \textit{Roumieu} and \textit{Beurling type}, respectively (see, e.g., \cite{GorV83,Gor-Knyaz,book}).
\end{defn}

In view of {\it Stirling's formula}, the 
sequence $\left\{ [n!]^\beta\right\}_{n=0}^\infty$ can be replaced with
$\left\{ n^{\beta n}\right\}_{n=0}^\infty$.

For $0\le\beta<\beta'<\infty$, the inclusions
\begin{equation*}
{\mathscr E}^{(\beta)}(A)\subseteq{\mathscr E}^{\{\beta\}}(A)
\subseteq {\mathscr E}^{(\beta')}(A)\subseteq
{\mathscr E}^{\{\beta'\}}(A)\subseteq C^{\infty}(A)
\end{equation*}
hold.

In particular, ${\mathscr E}^{\{1\}}(A)$ and ${\mathscr E}^{(1)}(A)$ are the classes of {\it analytic} and {\it entire} vectors of $A$, respectively \cite{Goodman,Nelson} and ${\mathscr E}^{\{0\}}(A)$ and ${\mathscr E}^{(0)}(A)$ are the classes of \textit{entire} vectors of $A$ of \textit{exponential} and \textit{minimal exponential type}, respectively (see, e.g., \cite{Radyno1983(1),Gor-Knyaz}).

In view of the \textit{closedness} of $A$, it is easily seen that the class ${\mathscr E}^{(1)}(A)$ forms the subspace of initial values in $X$ generating the (classical) solutions of \eqref{1}, which are entire vector functions represented by the power series
\[
\sum_{n=0}^\infty \dfrac{t^n}{n!}A^nf,\ t\ge 0,
f\in {\mathscr E}^{(1)}(A),
\]
the classes ${\mathscr E}^{\{\beta\}}(A)$ and ${\mathscr E}^{(\beta)}(A)$ with $0\le\beta<1$ being the subspaces of such initial values for which the solutions satisfy growth condition \eqref{order} with some (any) $\gamma>0$ and some $M>0$, respectively (cf. \cite{Levin}).

As is shown in \cite{GorV83} (see also \cite{Gor-Knyaz,book}), if $0<\beta<\infty$, for a {\it normal operator} $A$ in a complex Hilbert space,
\begin{equation}\label{GC}
{\mathscr E}^{\{\beta\}}(A)=\bigcup_{t>0} D(e^{t|A|^{1/\beta}})\ \text{and}\ 
{\mathscr E}^{(\beta)}(A)=\bigcap_{t>0} D(e^{t|A|^{1/\beta}}),
\end{equation}
the operator exponentials $e^{t|A|^{1/\beta}}$, $t>0$, understood in the sense of the Borel operational calculus (see, e.g., \cite{Dun-SchII,Plesner}).

In \cite{Markin2004(2),Markin2015}, descriptions \eqref{GC} are extended  to \textit{scalar type spectral operators} in a complex Banach space, in which form they are basic for our discourse. In \cite{Markin2015}, similar nature descriptions of the classes ${\mathscr E}^{\{0\}}(A)$ and ${\mathscr E}^{(0)}(A)$ ($\beta=0$), known for a normal operator $A$ in a complex Hilbert space (see, e.g., \cite{Gor-Knyaz}), are also generalized to scalar type spectral operators in a complex Banach space. In particular {\cite[Theorem $5.1$]{Markin2015}},
\[
{\mathscr E}^{\{0\}}(A)=\bigcup_{\alpha>0}E_A(\Delta_\alpha)X,
\] 
where
\begin{equation*}
\Delta_\alpha:=\left\{\lambda\in\C\,\middle|\,|\lambda|\le \alpha \right\},\ \alpha>0.
\end{equation*}

\section{Gevrey Ultradifferentiability 
of a Particular Weak Solution}

Here, we characterize a particular weak solution's of equation \eqref{1} with a scalar type spectral operator $A$ in a complex Banach space being strongly Gevrey ultradifferentiable on a subinterval $I$ of $[0,\infty)$.

\begin{prop}\label{particular}
Let $A$ be a scalar type spectral operator in a complex Banach space $(X,\|\cdot\|)$ with spectral measure $E_A(\cdot)$, $0\le \beta<\infty$, and $I$ be a subinterval of $[0,\infty)$. Then the restriction of
a weak solution $y(\cdot)$ of equation \eqref{1} to $I$ belongs to the Gevrey class ${\mathscr E}^{\{\beta\}}(I,X)$
\textup{(${\mathscr E}^{(\beta)}(I,X)$)} iff, for each $t\in I$, 
\begin{equation*}
y(t) \in {\mathscr E}^{\{\beta\}}(A)
\ \textup{(${\mathscr E}^{(\beta)}(A)$, respectively)},
\end{equation*}
in which case
\begin{equation*}
y^{(n)}(t)=A^ny(t),\ n\in \N,t\in I.
\end{equation*}
\end{prop}

\begin{proof}\ 

\textit{"Only if"} part.\quad Assume that a weak solution $y(\cdot)$ of \eqref{1} restricted to $I$ belongs to ${\mathscr E}^{\{\beta\}}(I,X)$ (${\mathscr E}^{(\beta)}(I,X)$).

This immediately implies that $y(\cdot)\in C^\infty(I,X)$. Whence, by Proposition \ref{Cor},
\[
y(t)\in C^\infty(A),\ t\in I,
\]
and
\[
y^{(n)}(t)=A^ny(t),\ n\in \N,t\in I.
\]

Furthermore, the fact that the restriction of $y(\cdot)$ to $I$ belongs to ${\mathscr E}^{\{\beta\}}(I,X)$ (${\mathscr E}^{(\beta)}(I,X)$)
implies that, for an arbitrary $t\in I$, some (any) $\alpha>0$, and some $c>0$:
\begin{equation*}
\|A^ny(t)\|=\|y^{(n)}(t)\|\le c\alpha^n[n!]^\beta,\ n=Z_+.
\end{equation*}
Therefore, for each $t\in I$,
\begin{equation*}
y(t) \in {\mathscr E}^{\{\beta\}}(A)\ ({\mathscr E}^{(\beta)}(A)).
\end{equation*}

\medskip
\textit{"If"} part.\quad
Let $y(\cdot)$ be a weak solution of equation \eqref{1} such that, for each $t\in I$,
\begin{equation*}
y(t) \in {\mathscr E}^{\{\beta\}}(A)\ ({\mathscr E}^{(\beta)}(A)).
\end{equation*}
Hence, for an arbitrary $t\in I$ and some (any) $\alpha>0$, 
there is a $c(t,\alpha)>0$ such that
\begin{equation}\label{const}
\|A^ny(t)\|\le c(t,\alpha)\alpha^n[n!]^\beta,\ n\in \Z_+.
\end{equation}

The inclusions
\begin{equation*}
{\mathscr E}^{(\beta)}(A)\subseteq
{\mathscr E}^{\{\beta\}}(A)\subseteq C^\infty(A)
\end{equation*}
imply by Proposition \ref{Cor} that 
\[
y(\cdot)\in C^\infty(I,X)
\]
and
\[
y^{(n)}(t)=A^ny(t),\ n\in \N,t\in I.
\]

By Theorem \ref{GWS},
\begin{equation*}
y(t)=e^{tA}f,\ t\ge 0,\ \text{with some}\
f \in \bigcap\limits_{t\ge 0} D(e^{tA}).
\end{equation*}

Fixing an arbitrary $[a,b]\subseteq I$, for every $n\in Z_+$, we have:
\begin{multline*}
\max_{a\le t\le b}\|y^{(n)}(t)\|=\max_{a\le t\le b}\|A^ny(t)\|=
\max_{a\le t\le b}\|A^ne^{tA}f\|
\\
\hfill
\text{by the properties of the \textit{o.c};}
\\
\shoveleft{
=\max_{a\le t\le b}
\left\|\int\limits_{\sigma(A)} \lambda ^n e^{t\lambda}\,d E_A(\lambda)f\right\|
\hfill
\text{as follows form the {\it Hahn-Banach Theorem};}
}\\
\shoveleft{
=\max_{a\le t\le b}
\sup_{g^*\in X^*,\,\|g^*\|=1}
\left|\left\langle
\int\limits_{\sigma(A)} \lambda ^n e^{t\lambda}\,dE_A(\lambda)f,g^*
\right\rangle\right|
\hfill
\text{by the properties of the \textit{o.c};}
}\\
\shoveleft{
=\max_{a\le t\le b}
\sup_{g^*\in X^*,\,\|g^*\|=1}
\left|
\int\limits_{\sigma(A)} \lambda ^n e^{t\lambda}\,d
\langle 
E_A(\lambda)f,g^*
\rangle\right|
}\\
\shoveleft{
\le \max_{a\le t\le b}
\sup_{g^*\in X^*,\,\|g^*\|=1}
\int\limits_{\sigma(A)} |\lambda|^n e^{t\Rep\lambda}\,dv(f,g^*,\lambda)
}\\
\shoveleft{
=\sup_{g^*\in X^*,\,\|g^*\|=1}
\sup_{a\le t\le b}
\biggl[\int\limits_{\{\lambda\in\sigma(A)|\Rep\lambda\le 0\}}|\lambda|^{n}e^{t\Rep\lambda}\,dv(f,g^*,\lambda)
}\\
\shoveleft{
+\int\limits_{\{\lambda\in\sigma(A)|\Rep\lambda>0\}}|\lambda|^{n}e^{t\Rep\lambda}\,dv(f,g^*,\lambda)
\biggr]
}\\
\shoveleft{
\le \sup_{g^*\in X^*,\,\|g^*\|=1}
\biggl[\int\limits_{\{\lambda\in\sigma(A)|\Rep\lambda\le 0\}}|\lambda|^{n}e^{a\Rep\lambda}\,dv(f,g^*,\lambda)
}\\
\shoveleft{
+\int\limits_{\{\lambda\in\sigma(A)|\Rep\lambda>0\}}|\lambda|^{n}e^{b\Rep\lambda}\,dv(f,g^*,\lambda)
\biggr]
}\\
\shoveleft{
\le \sup_{g^*\in X^*,\,\|g^*\|=1}
\left[\int\limits_{\sigma(A)}|\lambda|^{n}e^{a\Rep\lambda}\,dv(f,g^*,\lambda)
+\int\limits_{\sigma(A)} |\lambda|^{n}e^{b\Rep\lambda}\,dv(f,g^*,\lambda)
\right]
}\\
\hfill
\text{by the properties of the \textit{o.c} 
(see \eqref{power} and {\cite[Theorem XVIII.$2.11$ (f)]{Dun-SchIII}})
and \eqref{cond(i)};}
\\
\shoveleft{
\le \sup_{g^*\in X^*,\,\|g^*\|=1}
4M\left[\|A^ne^{aA}f\|+\|A^ne^{bA}f\|\right]\|g^*\|
\le 4M\left[\|A^ne^{aA}f\|+\|A^ne^{bA}f\|\right]
}\\
\ \
=4M\left[\|A^ny(a)\|+\|A^ny(b)\|\right]
=4M\left[\|y^{(n)}(a)\|+\|y^{(n)}(b)\|\right].
\hfill
\end{multline*}

Hence, in view of \eqref{const},
\begin{equation*}
\max_{a\le t\le b}\|y^{(n)}(t)\|\le 4M[c(a,\alpha)+c(b,\alpha)]
\max\left[\alpha(a),\alpha(b)\right]^n[n!]^\beta,
\ n\in \Z_+,
\end{equation*}
which implies that $y(\cdot)$ restricted to $I$
belongs to the Gevrey class ${\mathscr E}^{\{\beta\}}(I,X)$ (${\mathscr E}^{(\beta)}(I,X)$ completing the proof.
\end{proof}

Thus, we have obtained a generalization of 
{\cite[Proposition $3.1$]{Markin2001(1)}}, the counterpart for a normal operator $A$ in a complex Hilbert space. 

\section{Gevrey Ultradifferentiability of Weak Solutions}

In this section, we characterize the strong Gevrey ultradifferentiability of order $\beta\ge 1$ on $[0,\infty)$ of all weak solutions of equation \eqref{1} with a scalar type spectral operator $A$ in a complex Banach space.

\begin{thm}\label{closed}
Let $A$ be a scalar type spectral operator in a complex Banach space $(X,\|\cdot\|)$ with spectral measure $E_A(\cdot)$ and $ 1\le \beta<\infty$. Then the following statements are equivalent. 
\begin{itemize}
\item[\textup{(i)}] Every weak solution of equation \eqref{1} belongs to the $\beta$th-order Gevrey class 
${\mathscr E}^{(\beta )}\left([0,+\infty),X\right)$
of Beurling type.
\item[\textup{(ii)}] Every weak solution of equation \eqref{1} belongs to the $\beta$th-order Gevrey class 
${\mathscr E}^{\{\beta \}}\left([0,+\infty),X\right)$
of Roumieu type.
\item[\textup{(iii)}] There is a $b_+>0$ such that the set
$\sigma(A)\setminus {\mathscr P}^{\beta}_{b_+}$,
where
\begin{equation*}
{\mathscr P}^{\beta}_{b_+}:= \left\{ \lambda \in \C\, \middle|\,
\Rep\lambda \ge b_+|\Imp\lambda|^{1/\beta}\right\},
\end{equation*}
is bounded 
(see Fig. \ref{fig:graph3}).
\end{itemize}
\begin{figure}[h]
\centering
\includegraphics[height=1.6in]{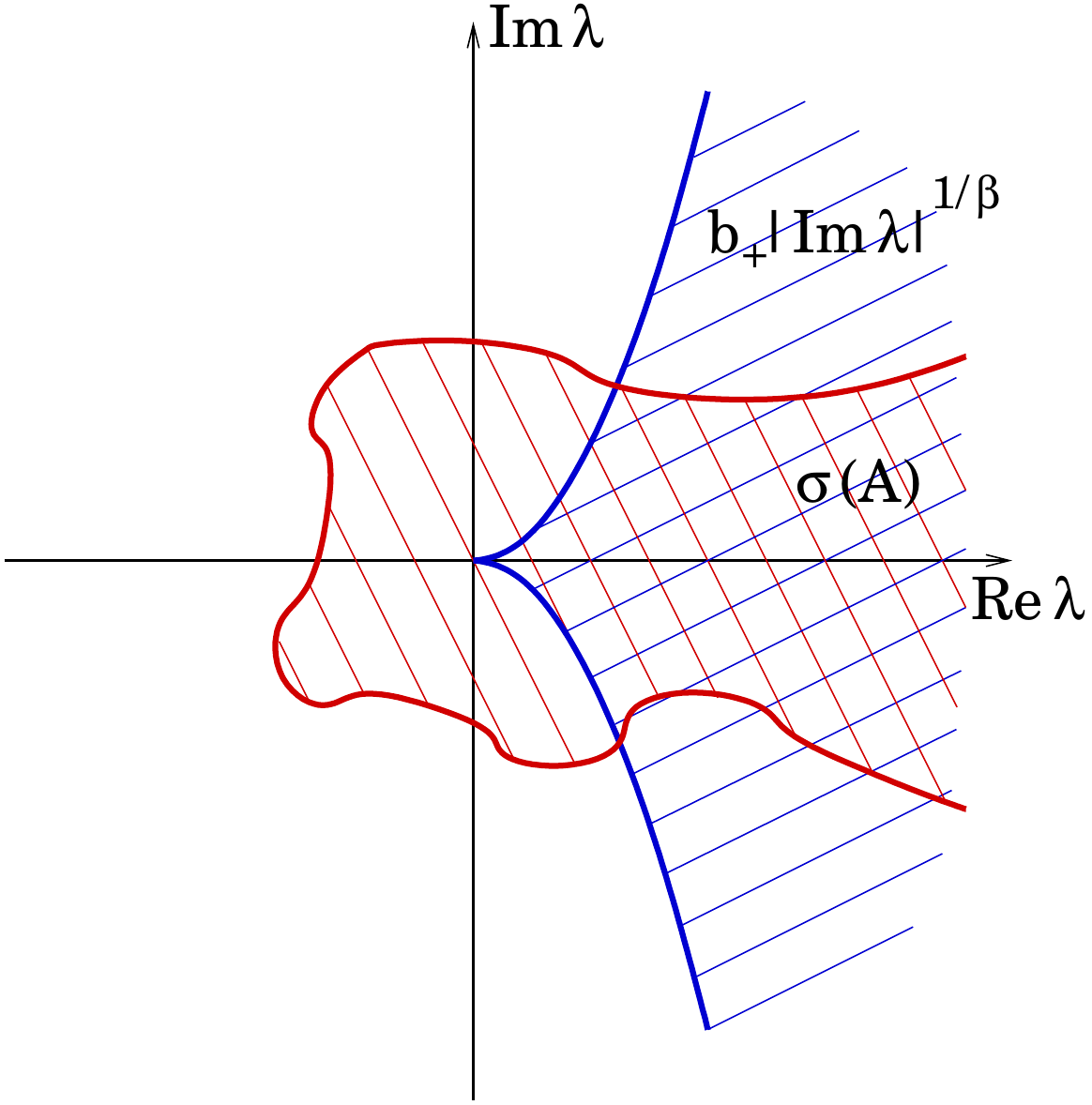}
\caption[]{}
\label{fig:graph3}
\end{figure}
\end{thm}

\begin{proof}\quad
We are to prove the closed chain of implications
\begin{equation*}
\text{(i)}\Rightarrow \text{(ii)}\Rightarrow
\text{(iii)}\Rightarrow
\text{(i)},
\end{equation*} 
the implication $\text{(i)}\Rightarrow \text{(ii)}$
following directly from the inclusion
\[
{\mathscr E}^{(\beta )}\left([0,+\infty),X\right)
\subseteq {\mathscr E}^{\{\beta\}}\left( [0,+\infty),X\right)
\]
(see Sec. \ref{GCF}).

To prove the implication $\text{(iii)}\Rightarrow \text{(i)}$, suppose that there is a $b_+>0$ such that the set $\sigma(A)\setminus{\mathscr P}_{b_+}^\beta$ 
is \textit{bounded} and let $y(\cdot)$ be an arbitrary weak solution of equation \eqref{1}. 

By Theorem \ref{GWS}, 
\begin{equation*}
y(t)=e^{tA}f,\ t\ge 0,\ \text{with some}\
f \in \bigcap_{t\ge 0}D(e^{tA}).
\end{equation*}

Our purpose is to show that $y(\cdot)\in {\mathscr E}^{(\beta )}\left([0,+\infty),X\right)$, which, by Proposition \ref{particular} and \eqref{GC}, is accomplished by showing that, for each $t\ge 0$,
\[
y(t)\in {\mathscr E}^{(\beta )}\left(A\right)
=\bigcap_{s>0} D(e^{s|A|^{1/\beta}}).
\]

Let us proceed by proving that, for any $t\ge 0$ and $s>0$,
\[
y(t)\in D(e^{s|A|^{1/\beta}})
\] 
via Proposition \ref{prop}.

For any $s>0$, $t\ge 0$ and an arbitrary $g^*\in X^*$,
\begin{multline}\label{first}
\int\limits_{\sigma(A)}e^{s|\lambda|^{1/\beta}}e^{t\Rep\lambda}\,dv(f,g^*,\lambda)
=\int\limits_{\sigma(A)\setminus{\mathscr P}_{b_+}^\beta}e^{s|\lambda|^{1/\beta}}e^{t\Rep\lambda}\,dv(f,g^*,\lambda)
\\
\shoveleft{
+\int\limits_{\left\{\lambda\in \sigma(A)\cap{\mathscr P}_{b_+}^\beta\,\middle|\,\Rep\lambda<1 \right\}}e^{s|\lambda|^{1/\beta}}e^{t\Rep\lambda}\,dv(f,g^*,\lambda)
}\\
\hspace{1.2cm}
+\int\limits_{\left\{\lambda\in \sigma(A)\cap{\mathscr P}_{b_+}^\beta\,\middle|\,\Rep\lambda\ge 1 \right\}}e^{s|\lambda|^{1/\beta}}e^{t\Rep\lambda}\,dv(f,g^*,\lambda)<\infty.
\hfill
\end{multline}

Indeed, 
\[
\int\limits_{\sigma(A)\setminus{\mathscr P}_{b_+}^\beta}e^{s|\lambda|^{1/\beta}}e^{t\Rep\lambda}\,dv(f,g^*,\lambda)<\infty
\]
and
\[
\int\limits_{\left\{\lambda\in \sigma(A)\cap{\mathscr P}_{b_+}^\beta\,\middle|\,\Rep\lambda<1 \right\}}e^{s|\lambda|^{1/\beta}}e^{t\Rep\lambda}\,dv(f,g^*,\lambda)<\infty
\]
due to the boundedness of the sets
\[
\sigma(A)\setminus{\mathscr P}_{b_+}^\beta\ \text{and}\
\left\{\lambda\in \sigma(A)\cap{\mathscr P}_{b_+}^\beta\;\middle|\;\Rep\lambda<1 \right\},
\]
the continuity of the integrated function 
on $\C$, and the finiteness of the measure $v(f,g^*,\cdot)$.

Further, for any $s>0$, $t\ge 0$ and an arbitrary $g^*\in X^*$,
\begin{multline}\label{interm}
\int\limits_{\left\{\lambda\in \sigma(A)\cap{\mathscr P}_{b_+}^\beta\,\middle|\,\Rep\lambda\ge 1 \right\}}e^{s|\lambda|^{1/\beta}}e^{t\Rep\lambda}\,dv(f,g^*,\lambda)
\\
\shoveleft{
\le\int\limits_{\left\{\lambda\in \sigma(A)\cap{\mathscr P}_{b_+}^\beta\,\middle|\,\Rep\lambda\ge 1 \right\}}e^{s\left[|\Rep\lambda|+|\Imp\lambda|\right]^{1/\beta}}e^{t\Rep\lambda}\,dv(f,g^*,\lambda)
}\\
\hfill
\text{since, for $\lambda\in\sigma(A)\cap{\mathscr P}_{b_+}^\beta$, $b_+^{-\beta}\Rep\lambda^\beta\ge |\Imp\lambda|$;}
\\
\shoveleft{
\le 
\int\limits_{\left\{\lambda\in \sigma(A)\cap{\mathscr P}_{b_+}^\beta\,\middle|\,\Rep\lambda\ge 1 \right\}}e^{s\left[\Rep\lambda+b_+^{-\beta}\Rep\lambda^\beta\right]^{1/\beta}}e^{t\Rep\lambda}\,dv(f,g^*,\lambda)
}\\
\hfill
\text{since, in view of $\Rep\lambda\ge 1$ and $\beta\ge 1$, $\Rep\lambda^\beta\ge\Rep\lambda$;}
\\
\shoveleft{
\le 
\int\limits_{\left\{\lambda\in \sigma(A)\cap{\mathscr P}_{b_+}^\beta\,\middle|\,\Rep\lambda\ge 1 \right\}}e^{s\left(1+b_+^{-\beta}\right)^{1/\beta}\Rep\lambda}e^{t\Rep\lambda}\,dv(f,g^*,\lambda)
}\\
\shoveleft{
= \int\limits_{\left\{\lambda\in \sigma(A)\cap{\mathscr P}_{b_+}^\beta\,\middle|\,\Rep\lambda\ge 1 \right\}}e^{\left[s\left(1+b_+^{-\beta}\right)^{1/\beta}+t\right]\Rep\lambda}\,dv(f,g^*,\lambda)
}\\
\hfill
\text{since $f\in \bigcap\limits_{t\ge 0}D(e^{tA})$, by Proposition \ref{prop};}
\\
\hspace{1.2cm}
<\infty. 
\hfill
\end{multline}

Also, for any $s>0$, $t\ge 0$ and an arbitrary $n\in\N$,
\begin{multline}\label{second}
\sup_{\{g^*\in X^*\,|\,\|g^*\|=1\}}
\int\limits_{\left\{\lambda\in\sigma(A)\,\middle|\,e^{s|\lambda|^{1/\beta}}e^{t\Rep\lambda}>n\right\}}
e^{s|\lambda|^{1/\beta}}e^{t\Rep\lambda}\,dv(f,g^*,\lambda)
\\
\shoveleft{
\le \sup_{\{g^*\in X^*\,|\,\|g^*\|=1\}}
\int\limits_{\left\{\lambda\in\sigma(A)\setminus{\mathscr P}_{b_+}^\beta\,\middle|\,e^{s|\lambda|^{1/\beta}}e^{t\Rep\lambda}>n\right\}}e^{s|\lambda|^{1/\beta}}e^{t\Rep\lambda}\,dv(f,g^*,\lambda)
}\\
\shoveleft{
+ \sup_{\{g^*\in X^*\,|\,\|g^*\|=1\}}
\int\limits_{\left\{\lambda\in\sigma(A)\cap{\mathscr P}_{b_+}^\beta\,\middle|\,\Rep\lambda<1,\, e^{s|\lambda|^{1/\beta}}e^{t\Rep\lambda}>n\right\}}e^{s|\lambda|^{1/\beta}}e^{t\Rep\lambda}\,dv(f,g^*,\lambda)
}\\
\shoveleft{
+ \sup_{\{g^*\in X^*\,|\,\|g^*\|=1\}}
\int\limits_{\left\{\lambda\in\sigma(A)\cap{\mathscr P}_{b_+}^\beta\,\middle|\,\Rep\lambda\ge 1,\, e^{s|\lambda|^{1/\beta}}e^{t\Rep\lambda}>n\right\}}e^{s|\lambda|^{1/\beta}}e^{t\Rep\lambda}\,dv(f,g^*,\lambda)
}\\
\hspace{1.2cm}
\to 0,\ n\to\infty.
\hfill
\end{multline}

Indeed, since, due to the boundedness of the sets
\[
\sigma(A)\setminus{\mathscr P}_{b_+}^\beta\ \text{and}\
\left\{\lambda\in\sigma(A)\cap{\mathscr P}_{b_+}^\beta\,\middle|\,\Rep\lambda<1\right\}
\]
and the continuity of the integrated function on $\C$,
the sets
\[
\left\{\lambda\in\sigma(A)\setminus{\mathscr P}_{b_+}^\beta\,\middle|\,e^{s|\lambda|^{1/\beta}}e^{t\Rep\lambda}>n\right\}
\]
and 
\[
\left\{\lambda\in\sigma(A)\cap{\mathscr P}_{b_+}^\beta\,\middle|\,\Rep\lambda<1,\, e^{s|\lambda|^{1/\beta}}e^{t\Rep\lambda}>n\right\}
\]
are \textit{empty} for all sufficiently large $n\in \N$,
we immediately infer that, for any $s>0$ and $t\ge 0$,
\[
\lim_{n\to\infty}\sup_{\{g^*\in X^*\,|\,\|g^*\|=1\}}
\int\limits_{\left\{\lambda\in\sigma(A)\setminus{\mathscr P}_{b_+}^\beta\,\middle|\,e^{s|\lambda|^{1/\beta}}e^{t\Rep\lambda}>n\right\}}e^{s|\lambda|^{1/\beta}}e^{t\Rep\lambda}\,dv(f,g^*,\lambda)=0
\]
and
\[
\lim_{n\to\infty}\sup_{\{g^*\in X^*\,|\,\|g^*\|=1\}}
\int\limits_{\left\{\lambda\in\sigma(A)\cap{\mathscr P}_{b_+}^\beta\,\middle|\,\Rep\lambda<1,\, e^{s|\lambda|^{1/\beta}}e^{t\Rep\lambda}>n\right\}}e^{s|\lambda|^{1/\beta}}e^{t\Rep\lambda}\,dv(f,g^*,\lambda)
=0.
\]

Further, for any $s>0$, $t\ge 0$, and an arbitrary $n\in\N$, 
\begin{multline*}
\sup_{\{g^*\in X^*\,|\,\|g^*\|=1\}}
\int\limits_{\left\{\lambda\in\sigma(A)\cap{\mathscr P}_{b_+}^\beta\,\middle|\,\Rep\lambda\ge 1,\, e^{s|\lambda|^{1/\beta}}e^{t\Rep\lambda}>n\right\}}e^{s|\lambda|^{1/\beta}}e^{t\Rep\lambda}\,dv(f,g^*,\lambda)
\\
\hfill
\text{as in \eqref{interm};}
\\
\shoveleft{
\le \sup_{\{g^*\in X^*\,|\,\|g^*\|=1\}}
\int\limits_{\left\{\lambda\in\sigma(A)\cap{\mathscr P}_{b_+}^\beta\,\middle|\,\Rep\lambda\ge 1,\, e^{s|\lambda|^{1/\beta}}e^{t\Rep\lambda}>n\right\}}e^{\left[s\left(1+b_+^{-\beta}\right)^{1/\beta}+t\right]\Rep\lambda}\,dv(f,g^*,\lambda)
}\\
\hfill
\text{since $f\in \bigcap\limits_{t\ge 0}D(e^{tA})$, by \eqref{cond(ii)};}
\\
\shoveleft{
\le \sup_{\{g^*\in X^*\,|\,\|g^*\|=1\}}
}\\
\shoveleft{
4M\left\|E_A\left(\left\{\lambda\in\sigma(A)\cap{\mathscr P}_{b_+}^\beta\,\middle|\,\Rep\lambda\ge 1,\, e^{s|\lambda|^{1/\beta}}e^{t\Rep\lambda}>n\right\}\right)
e^{\left[s\left(1+b_+^{-\beta}\right)^{1/\beta}+t\right]A}f\right\|\|g^*\|
}\\
\shoveleft{
\le 4M\left\|E_A\left(\left\{\lambda\in\sigma(A)\cap{\mathscr P}_{b_+}^\beta\,\middle|\,\Rep\lambda\ge 1,\, e^{s|\lambda|^{1/\beta}}e^{t\Rep\lambda}>n\right\}\right)
e^{\left[s\left(1+b_+^{-\beta}\right)^{1/\beta}+t\right]A}f\right\|
}\\
\hfill
\text{by the strong continuity of the {\it s.m.};}
\\
\ \
\to 4M\left\|E_A\left(\emptyset\right)e^{\left[s\left(1+b_+^{-\beta}\right)^{1/\beta}+t\right]A}f\right\|=0,\ n\to\infty.
\hfill
\end{multline*}

By Proposition \ref{prop} and the properties of the \textit{o.c.} (see {\cite[Theorem XVIII.$2.11$ (f)]{Dun-SchIII}}), \eqref{first} and \eqref{second} jointly imply that, for any $t\ge 0$ and $s>0$,
\[
f\in D(e^{s|A|^{1/\beta}}e^{tA}).
\]

In view of \eqref{GC}, the latter implies that, for each $t\ge 0$, 
\begin{equation*}
y(t)=e^{tA}f\in \bigcap_{s>0} D(e^{s|A|^{1/\beta}})
={\mathscr E}^{(\beta )}(A).
\end{equation*}

Whence, by Proposition \ref{particular}, we infer that
\begin{equation*}
y(\cdot) \in {\mathscr E}^{(\beta )}([0,\infty),X),
\end{equation*}
which completes the proof of the implication $\text{(iii)}\Rightarrow \text{(i)}$.

\medskip
Let us prove the remaining implication $\text{(ii)}\Rightarrow \text{(iii)}$ {\it by contrapositive} assuming that, for any $b_+>0$, the set 
$\sigma(A)\setminus {\mathscr P}_{b_+}^\beta$ is \textit{unbounded}. In particular, this means that, for any $n\in \N$, unbounded is the set
\begin{equation*}
\sigma(A)\setminus {\mathscr P}^\beta_{n^{-2}}=
\left\{ \lambda \in \sigma(A)\,\middle| 
\Rep\lambda < n^{-2}|\Imp\lambda|^{1/\beta}\right\}.
\end{equation*} 

Hence, we can choose a sequence of points $\left\{\lambda_n\right\}_{n=1}^\infty$ 
in the complex plane as follows:
\begin{equation*}
\begin{split}
&\lambda_n \in \sigma(A),\ n\in \N,\\
&\Rep\lambda_n <n^{-2}|\Imp\lambda_n|^{1/\beta},\ n\in \N,\\
&\lambda_0:=0,\ |\lambda_n|>\max\left[n,|\lambda_{n-1}|\right],\ n\in \N.\\
\end{split}
\end{equation*}

The latter implies, in particular, that the points $\lambda_n$, $n\in\N$, are \textit{distinct} ($\lambda_i \neq \lambda_j$, $i\neq j$).

Since, for each $n\in \N$, the set
\begin{equation*}
\left\{ \lambda \in {\mathbb C}\,\middle|\, 
\Rep\lambda <n^{-2}|\Imp\lambda|^{1/\beta},\
|\lambda|>\max\bigl[n,|\lambda_{n-1}|\bigr]\right\}
\end{equation*}
is {\it open} in $\C$, along with the point $\lambda_n$, it contains the {\it open disk}
\begin{equation*}
\Delta_n:=\left\{\lambda \in \C\, \middle|\,|\lambda-\lambda_n|<\varepsilon_n \right\}
\end{equation*} 
of some radius $\varepsilon_n>0$, i.e., for each $\lambda \in \Delta_n$,
\begin{equation}\label{disks1}
\Rep\lambda < n^{-2}|\Imp\lambda|^{1/\beta}\ \text{and}\ |\lambda|>\max\bigl[n,|\lambda_{n-1}|\bigr].
\end{equation}

Furthermore, under the circumstances, we can regard the radii of the disks to be small enough so that
\begin{equation}\label{radii1}
\begin{split}
&0<\varepsilon_n<\dfrac{1}{n},\ n\in\N,\ \text{and}\\
&\Delta_i \cap \Delta_j=\emptyset,\ i\neq j
\quad \text{(i.e., the disks are {\it pairwise disjoint})}.
\end{split}
\end{equation}

Whence, by the properties of the {\it s.m.}, 
\begin{equation*}
E_A(\Delta_i)E_A(\Delta_j)=0,\ i\neq j,
\end{equation*}
where $0$ stands for the \textit{zero operator} on $X$.

Observe also, that the subspaces $E_A(\Delta_n)X$, $n\in \N$, are \textit{nontrivial} since
\[
\Delta_n \cap \sigma(A)\neq \emptyset,\ n\in\N,
\]
with $\Delta_n$ being an \textit{open set} in $\C$. 

By choosing a unit vector $e_n\in E_A(\Delta_n)X$, $n\in\N$, we obtain a vector sequence 
$\left\{e_n\right\}_{n=1}^\infty$ such that
\begin{equation}\label{ortho1}
\|e_n\|=1,\ n\in\N,\ \text{and}\ E_A(\Delta_i)e_j=\delta_{ij}e_j,\ i,j\in\N,
\end{equation}
where $\delta_{ij}$ is the \textit{Kronecker delta}.

As is easily seen, \eqref{ortho1} implies that the vectors $e_n$, $n\in \N$, are \textit{linearly independent}.

Furthermore, there is an $\varepsilon>0$ such that
\begin{equation}\label{dist1}
d_n:=\dist\left(e_n,\spa\left(\left\{e_i\,|\,i\in\N,\ i\neq n\right\}\right)\right)\ge\varepsilon,\ n\in\N.
\end{equation}

Indeed, the opposite implies the existence of a subsequence $\left\{d_{n(k)}\right\}_{k=1}^\infty$ such that
\begin{equation*}
d_{n(k)}\to 0,\ k\to\infty.
\end{equation*}

Then, by selecting a vector
\[
f_{n(k)}\in 
\spa\left(\left\{e_i\,|\,i\in\N,\ i\neq n(k)\right\}\right),\ k\in\N,
\] 
such that 
\[
\|e_{n(k)}-f_{n(k)}\|<d_{n(k)}+1/k,\ k\in\N,
\]
we arrive at
\begin{multline*}
1=\|e_{n(k)}\|
\hfill
\text{since, by \eqref{ortho1}, 
$E_A(\Delta_{n(k)})f_{n(k)}=0$;}
\\
\shoveleft{
=\|E_A(\Delta_{n(k)})(e_{n(k)}-f_{n(k)})\|\
\le \|E_A(\Delta_{n(k)})\|\|e_{n(k)}-f_{n(k)}\|
\hfill
\text{by \eqref{bounded};}
}\\
\ \
\le M\|e_{n(k)}-f_{n(k)}\|\le M\left[d_{n(k)}+1/k\right]
\to 0,\ k\to\infty,
\hfill
\end{multline*}
which is a \textit{contradiction} proving \eqref{dist1}. 

As follows from the {\it Hahn-Banach Theorem}, for any $n\in\N$, there is an $e^*_n\in X^*$ such that 
\begin{equation}\label{H-B1}
\|e_n^*\|=1,\ n\in\N,\ \text{and}\ \langle e_i,e_j^*\rangle=\delta_{ij}d_i,\ i,j\in\N.
\end{equation}

Let us consider separately the two possibilities concerning the sequence of the real parts $\{\Rep\lambda_n\}_{n=1}^\infty$: its being \textit{bounded above} or \textit{unbounded above}. 

First, suppose that the sequence $\{\Rep\lambda_n\}_{n=1}^\infty$ is \textit{bounded above}, i.e., there is such an $\omega>0$ that
\begin{equation}\label{bounded1}
\Rep\lambda_n \le \omega,\ n\in\N,
\end{equation}
and consider the element
\begin{equation*}
f:=\sum_{k=1}^\infty k^{-2}e_k\in X,
\end{equation*}
which is well defined since $\left\{k^{-2}\right\}_{k=1}^\infty\in l_1$ ($l_1$ is the space of absolutely summable sequences) and $\|e_k\|=1$, $k\in\N$ (see \eqref{ortho1}).

In view of \eqref{ortho1}, by the properties of the \textit{s.m.},
\begin{equation}\label{vectors1}
E_A(\cup_{k=1}^\infty\Delta_k)f=f\ \text{and}\ E_A(\Delta_k)f=k^{-2}e_k,\ k\in\N.
\end{equation}

For any $t\ge 0$ and an arbitrary $g^*\in X^*$,
\begin{multline}\label{first1}
\int\limits_{\sigma(A)}e^{t\Rep\lambda}\,dv(f,g^*,\lambda)
\hfill \text{by \eqref{vectors1};}
\\
\shoveleft{
=\int\limits_{\sigma(A)} e^{t\Rep\lambda}\,d v(E_A(\cup_{k=1}^\infty \Delta_k)f,g^*,\lambda)
\hfill
\text{by \eqref{decompose};}
}\\
\shoveleft{
=\sum_{k=1}^\infty\int\limits_{\sigma(A)\cap\Delta_k}e^{t\Rep\lambda}\,dv(E_A(\Delta_k)f,g^*,\lambda)
\hfill 
\text{by \eqref{vectors1};}
}\\
\shoveleft{
=\sum_{k=1}^\infty k^{-2}\int\limits_{\sigma(A)\cap\Delta_k}e^{t\Rep\lambda}\,dv(e_k,g^*,\lambda)
}\\
\hfill
\text{since, for $\lambda\in \Delta_k$, by \eqref{bounded1} and \eqref{radii1},}\ 
\Rep\lambda=\Rep\lambda_k+(\Rep\lambda-\Rep\lambda_k)
\\
\hfill
\le \Rep\lambda_k+|\lambda-\lambda_k|\le \omega+\varepsilon_k\le \omega+1;
\\
\shoveleft{
\le e^{t(\omega+1)}\sum_{k=1}^\infty k^{-2}\int\limits_{\sigma(A)\cap\Delta_k}1\,dv(e_k,g^*,\lambda)
= e^{t(\omega+1)}\sum_{k=1}^\infty k^{-2}v(e_k,g^*,\Delta_k)
}\\
\hfill
\text{by \eqref{tv};}
\\
\hspace{1.2cm}
\le e^{t(\omega+1)}\sum_{k=1}^\infty k^{-2}4M\|e_k\|\|g^*\|
= 4Me^{t(\omega+1)}\|g^*\|\sum_{k=1}^\infty k^{-2}<\infty.
\hfill
\end{multline} 

Similarly, for any $t\ge 0$ and an arbitrary $n\in\N$,
\begin{multline}\label{second1}
\sup_{\{g^*\in X^*\,|\,\|g^*\|=1\}}
\int\limits_{\left\{\lambda\in\sigma(A)\,\middle|\,e^{t\Rep\lambda}>n\right\}} 
e^{t\Rep\lambda}\,dv(f,g^*,\lambda)
\\
\shoveleft{
\le 
\sup_{\{g^*\in X^*\,|\,\|g^*\|=1\}}e^{t(\omega+1)}\sum_{k=1}^\infty k^{-2}
\int\limits_{\left\{\lambda\in\sigma(A)\,\middle|\,e^{t\Rep\lambda}>n\right\}\cap \Delta_k}1\,dv(e_k,g^*,\lambda) 
}\\
\hfill \text{by \eqref{vectors1};}
\\
\shoveleft{
=e^{t(\omega+1)}\sup_{\{g^*\in X^*\,|\,\|g^*\|=1\}}\sum_{k=1}^\infty 
\int\limits_{\left\{\lambda\in\sigma(A)\,\middle|\,e^{t\Rep\lambda}>n\right\}\cap \Delta_k}1\,dv(E_A(\Delta_k)f,g^*,\lambda) 
}\\
\hfill \text{by \eqref{decompose};}
\\
\shoveleft{
= e^{t(\omega+1)}\sup_{\{g^*\in X^*\,|\,\|g^*\|=1\}}
\int\limits_{\{\lambda\in\sigma(A)\,|\,e^{t\Rep\lambda}>n\}}1\,dv(E_A(\cup_{k=1}^\infty\Delta_k)f,g^*,\lambda)
}\\
\hfill \text{by \eqref{vectors1};}
\\
\shoveleft{
= e^{t(\omega+1)}\sup_{\{g^*\in X^*\,|\,\|g^*\|=1\}}
\int\limits_{\{\lambda\in\sigma(A)\,|\,e^{t\Rep\lambda}>n\}}1\,dv(f,g^*,\lambda)
\hfill
\text{by \eqref{cond(ii)};}
}\\
\shoveleft{
\le e^{t(\omega+1)}\sup_{\{g^*\in X^*\,|\,\|g^*\|=1\}}4M\left\|E_A\left(\left\{\lambda\in\sigma(A)\,\middle|\,e^{t\Rep\lambda}>n\right\}\right)f\right\|\|g^*\|
}\\
\shoveleft{
\le 4Me^{t(\omega+1)}\left\|E_A\left(\left\{\lambda\in\sigma(A)\,\middle|\,e^{t\Rep\lambda}>n\right\}\right)f\right\|
}\\
\hfill
\text{by the strong continuity of the {\it s.m.};}
\\
\hspace{1.2cm}
\to 4Me^{t(\omega+1)}\left\|E_A\left(\emptyset\right)f\right\|=0,\ n\to\infty.
\hfill
\end{multline}

By Proposition \ref{prop}, \eqref{first1} and \eqref{second1} jointly imply that
\[
f\in \bigcap\limits_{t\ge 0}D(e^{tA}),
\] 
and hence,
by Theorem \ref{GWS},
\[
y(t):=e^{tA}f,\ t\ge 0,
\]
is a weak solution of equation \eqref{1}.

Let
\begin{equation}\label{functional1}
h^*:=\sum_{k=1}^\infty k^{-2}e_k^*\in X^*,
\end{equation}
the functional being well defined since $\{k^{-2}\}_{k=1}^\infty\in l_1$ and $\|e_k^*\|=1$, $k\in\N$ (see \eqref{H-B1}).

In view of \eqref{H-B1} and \eqref{dist1}, we have:
\begin{equation}\label{funct-dist1}
\langle e_k,h^*\rangle=\langle e_k,k^{-2}e_k^*\rangle=d_k k^{-2}\ge \varepsilon k^{-2},\ k\in\N.
\end{equation}

For any $s>0$,
\begin{multline*}
\int\limits_{\sigma(A)}e^{s|\lambda|^{1/\beta}}\,dv(f,h^*,\lambda)
\hfill
\text{by \eqref{decompose} as in \eqref{first1};}
\\
\shoveleft{
=\sum_{k=1}^\infty k^{-2}\int\limits_{\sigma(A)\cap\Delta_k}e^{s|\lambda|^{1/\beta}}\,dv(e_k,h^*,\lambda)
\hfill
\text{since, for $\lambda\in \Delta_k$, by \eqref{disks1}, $|\lambda|\ge k$;}
}\\
\shoveleft{
\ge
\sum_{k=1}^\infty k^{-2}e^{sk^{1/\beta}}\int\limits_{\sigma(A)\cap\Delta_k}1\,dv(e_k,h^*,\lambda)
= \sum_{k=1}^\infty k^{-2}e^{sk^{1/\beta}} v(e_k,h^*,\Delta_k)
}\\
\shoveleft{
\ge\sum_{k=1}^\infty k^{-2}e^{sk^{1/\beta}}|\langle E_A(\Delta_k)e_k,h^*\rangle|
\hfill
\text{by \eqref{ortho1} and \eqref{funct-dist1};}
}\\
\ \
\ge \sum_{k=1}^\infty \varepsilon k^{-4}e^{sk^{1/\beta}}=\infty.
\hfill
\end{multline*} 

Whence, by Proposition \ref{prop} and \eqref{GC}, we infer that
\[
y(0)=f\notin \bigcup_{s>0} D(e^{s|A|^{1/\beta}})
={\mathscr E}^{\{\beta\}}(A)
\]
which, by Proposition \ref{particular}, implies that the weak solution $y(t)=e^{tA}f$, $t\ge 0$, 
of equation \eqref{1} does not belong to the 
Gevrey class ${\mathscr E}^{\{\beta \}}\left( [0,+\infty),X\right)$ of Roumieu
type and completes our consideration of the case of
the sequence's $\{\Rep\lambda_n\}_{n=1}^\infty$ being \textit{bounded above}. 

Now, suppose that the sequence $\{\Rep\lambda_n\}_{n=1}^\infty$
is \textit{unbounded above}. 

Therefore, there is a subsequence $\{\Rep\lambda_{n(k)}\}_{k=1}^\infty$ such that
\begin{equation}\label{infinity}
\Rep\lambda_{n(k)} \ge k,\ k\in\N.
\end{equation}

Consider the elements
\begin{equation*}
f:=\sum_{k=1}^\infty e^{-n(k)\Rep\lambda_{n(k)}}e_{n(k)}\in X
\ \text{and}\ h:=\sum_{k=1}^\infty e^{-\frac{n(k)}{2}\Rep\lambda_{n(k)}}e_{n(k)}\in X,
\end{equation*}
well defined since, by \eqref{infinity},
\[
\left\{e^{-n(k)\Rep\lambda_{n(k)}}\right\}_{k=1}^\infty,
\left\{e^{-\frac{n(k)}{2}\Rep\lambda_{n(k)}}\right\}_{k=1}^\infty
\in l_1
\]
and $\|e_{n(k)}\|=1$, $k\in\N$ (see \eqref{ortho1}).

By \eqref{ortho1},
\begin{equation}\label{subvectors1}
E_A(\cup_{k=1}^\infty\Delta_{n(k)})f=f\ \text{and}\
E_A(\Delta_{n(k)})f=e^{-n(k)\Rep\lambda_{n(k)}}e_{n(k)},\
k\in\N,
\end{equation}
and
\begin{equation}\label{subvectors2}
E_A(\cup_{k=1}^\infty\Delta_{n(k)})h=h\ \text{and}\
E_A(\Delta_{n(k)})h=e^{-\frac{n(k)}{2}\Rep\lambda_{n(k)}}e_{n(k)},\ k\in\N.
\end{equation}

For any $t\ge 0$ and an arbitrary $g^*\in X^*$, 
\begin{multline}\label{first2}
\int\limits_{\sigma(A)}e^{t\Rep\lambda}\,dv(f,g^*,\lambda)
\hfill
\text{by \eqref{decompose} as in \eqref{first1};}
\\
\shoveleft{
=\sum_{k=1}^\infty e^{-n(k)\Rep\lambda_{n(k)}}\int\limits_{\sigma(A)\cap\Delta_{n(k)}}e^{t\Rep\lambda}\,dv(e_{n(k)},g^*,\lambda)
}\\
\hfill
\text{since, for $\lambda\in \Delta_{n(k)}$, by \eqref{radii1},}\ \Rep\lambda
=\Rep\lambda_{n(k)}+(\Rep\lambda-\Rep\lambda_{n(k)})
\\
\hfill
\le \Rep\lambda_{n(k)}+|\lambda-\lambda_{n(k)}|\le \Rep\lambda_{n(k)}+1;
\\
\shoveleft{
\le \sum_{k=1}^\infty e^{-n(k)\Rep\lambda_{n(k)}}
e^{t(\Rep\lambda_{n(k)}+1)}
\int\limits_{\sigma(A)\cap\Delta_{n(k)}}1\,dv(e_{n(k)},g^*,\lambda)
}\\
\shoveleft{
= e^t\sum_{k=1}^\infty e^{-[n(k)-t]\Rep\lambda_{n(k)}}v(e_{n(k)},g^*,\Delta_{n(k)})
\hfill
\text{by \eqref{tv};}
}\\
\shoveleft{
\le e^t\sum_{k=1}^\infty e^{-[n(k)-t]\Rep\lambda_{n(k)}}4M\|e_{n(k)}\|\|g^*\|
= 4Me^t\|g^*\|\sum_{k=1}^\infty e^{-[n(k)-t]\Rep\lambda_{n(k)}}
}\\
\hspace{1.2cm}
<\infty.
\hfill
\end{multline}

Indeed, for all $k\in \N$ sufficiently large so that
\[
n(k)\ge t+1,
\]
in view of \eqref{infinity}, 
\[
e^{-[n(k)-t]\Rep\lambda_{n(k)}}\le e^{-k}.
\]

Similarly, for any $t\ge 0$ and an arbitrary,
\begin{multline}\label{second2}
\sup_{\{g^*\in X^*\,|\,\|g^*\|=1\}}
\int\limits_{\left\{\lambda\in\sigma(A)\,\middle|\,e^{t\Rep\lambda}>n\right\}}e^{t\Rep\lambda}\,dv(f,g^*,\lambda)
\\
\shoveleft{
\le \sup_{\{g^*\in X^*\,|\,\|g^*\|=1\}}e^t\sum_{k=1}^\infty e^{-[n(k)-t]\Rep\lambda_{n(k)}}
\int\limits_{\left\{\lambda\in\sigma(A)\,\middle|\,e^{t\Rep\lambda}>n\right\}\cap \Delta_{n(k)}}1\,dv(e_{n(k)},g^*,\lambda)
}\\
\shoveleft{
=e^t\sup_{\{g^*\in X^*\,|\,\|g^*\|=1\}}\sum_{k=1}^\infty e^{-\left[\frac{n(k)}{2}-t\right]\Rep\lambda_{n(k)}}
e^{-\frac{n(k)}{2}\Rep\lambda_{(k)}}
}\\
\shoveleft{
\int\limits_{\left\{\lambda\in\sigma(A)\,\middle|\,e^{t\Rep\lambda}>n\right\}\cap \Delta_{n(k)}}1\,dv(e_{n(k)},g^*,\lambda)
}\\
\hfill
\text{since, by \eqref{infinity}, there is an $L>0$ such that
$e^{-\left[\frac{n(k)}{2}-t\right]\Rep\lambda_{n(k)}}\le L$, $k\in\N$;}
\\
\shoveleft{
\le Le^t\sup_{\{g^*\in X^*\,|\,\|g^*\|=1\}}\sum_{k=1}^\infty e^{-\frac{n(k)}{2}\Rep\lambda_{n(k)}}
\int\limits_{\left\{\lambda\in\sigma(A)\,\middle|\,e^{t\Rep\lambda}>n\right\}\cap \Delta_{n(k)}}1\,dv(e_{n(k)},g^*,\lambda)
}\\
\hfill
\text{by \eqref{subvectors2};}
\\
\shoveleft{
= Le^t\sup_{\{g^*\in X^*\,|\,\|g^*\|=1\}}\sum_{k=1}^\infty
\int\limits_{\left\{\lambda\in\sigma(A)\,\middle|\,e^{t\Rep\lambda}>n\right\}\cap \Delta_{n(k)}}1\,dv(E_A(\Delta_{n(k)})h,g^*,\lambda)
}\\
\hfill
\text{by \eqref{decompose};}
\\
\shoveleft{
= Le^t\sup_{\{g^*\in X^*\,|\,\|g^*\|=1\}}
\int\limits_{\left\{\lambda\in\sigma(A)\,\middle|\,e^{t\Rep\lambda}>n\right\}}1\,dv(E_A(\cup_{k=1}^\infty\Delta_{n(k)})h,g^*,\lambda)
}\\
\hfill
\text{by \eqref{subvectors2};}
\\
\shoveleft{
=Le^t\sup_{\{g^*\in X^*\,|\,\|g^*\|=1\}}\int\limits_{\{\lambda\in\sigma(A)\,|\,e^{t\Rep\lambda}>n\}}1\,dv(h,g^*,\lambda)
\hfill
\text{by \eqref{cond(ii)};}
}\\
\shoveleft{
\le Le^t\sup_{\{g^*\in X^*\,|\,\|g^*\|=1\}}4M
\left\|E_A\left(\left\{\lambda\in\sigma(A)\,\middle|\,e^{t\Rep\lambda}>n\right\}\right)h\right\|\|g^*\|
}\\
\shoveleft{
\le 4LMe^t\|E_A(\{\lambda\in\sigma(A)\,|\,e^{t\Rep\lambda}>n\})h\|
}\\
\hfill
\text{by the strong continuity of the {\it s.m.};}
\\
\hspace{1.2cm}
\to 4LMe^t\left\|E_A\left(\emptyset\right)h\right\|=0,\ n\to\infty.
\hfill
\end{multline}

By Proposition \ref{prop}, \eqref{first2} and \eqref{second2} jointly imply that
\[
f\in \bigcap\limits_{t\ge 0}D(e^{tA}),
\] 
and hence,
by Theorem \ref{GWS},
\[
y(t):=e^{tA}f,\ t\ge 0,
\]
is a weak solution of equation \eqref{1}.

Since, for any $\lambda \in \Delta_{n(k)}$, $k\in \N$, by \eqref{radii1}, \eqref{infinity},
\begin{multline*}
\Rep\lambda =\Rep\lambda_{n(k)}-(\Rep\lambda_{n(k)}-\Rep\lambda)
\ge
\Rep\lambda_{n(k)}-|\Rep\lambda_{n(k)}-\Rep\lambda|
\\
\ \ \
\ge 
\Rep\lambda_{n(k)}-\varepsilon_{n(k)}
\ge \Rep\lambda_{n(k)}-1/n(k)\ge k-1\ge 0
\hfill
\end{multline*}
and, by \eqref{disks1},
\[
\Rep\lambda<n(k)^{-2}|\Imp\lambda|^{1/\beta},
\]
we infer that, for any $\lambda \in \Delta_{n(k)}$, $k\in \N$,
\begin{equation*}
|\lambda|\ge|\Imp\lambda|\ge 
\left[n(k)^2\Rep\lambda\right]^\beta\ge \left[n(k)^2(\Rep\lambda_{n(k)}-1/n(k))\right]^\beta.
\end{equation*}

Using this estimate, for an arbitrary $s>0$ and the functional $h^*\in X^*$ defined by \eqref{functional1}, we have:
\begin{multline}\label{notin}
\int\limits_{\sigma(A)}e^{s|\lambda|^{1/\beta}}\,dv(f,h^*,\lambda)
\hfill
\text{by \eqref{decompose} as in \eqref{first1};}
\\
\shoveleft{
=\sum_{k=1}^\infty e^{-n(k)\Rep\lambda_{n(k)}}\int\limits_{\sigma(A)\cap\Delta_{n(k)}}e^{s|\lambda|^{1/\beta}}\,dv(e_{n(k)},h^*,\lambda)
}\\
\shoveleft{
\ge\sum_{k=1}^\infty e^{-n(k)\Rep\lambda_{n(k)}}e^{sn(k)^2(\Rep\lambda_{n(k)}-1/n(k))}v(e_{n(k)},h^*,\Delta_{n(k)})
}\\
\shoveleft{
\ge \sum_{k=1}^\infty e^{-n(k)\Rep\lambda_{n(k)}}e^{sn(k)^2(\Rep\lambda_{n(k)}-1/n(k))}|\langle E_A(\Delta_{n(k)})e_{n(k)},h^*\rangle|
}\\
\hfill
\text{by \eqref{ortho1} and \eqref{funct-dist1};}
\\
\hspace{1.2cm}
\ge \sum_{k=1}^\infty \varepsilon
e^{(sn(k)-1)n(k)\Rep\lambda_{n(k)}-sn(k)}n(k)^{-2}
=\infty.
\hfill
\end{multline} 

Indeed, for all $k\in\N$ sufficiently large so that 
\begin{equation*}
sn(k)\ge s+2,
\end{equation*}
in view of \eqref{infinity},
\begin{multline*}
e^{(sn(k)-1)n(k)\Rep\lambda_{n(k)}-sn(k)}n(k)^{-2}
\ge 
e^{(s+1)n(k)-sn(k)}n(k)^{-2}=e^{n(k)}n(k)^{-2}
\\
\ \
\to\infty,\ k\to\infty.
\hfill
\end{multline*}

By Proposition \ref{prop} and \eqref{GC},
\eqref{notin} implies that
\[
y(0)=f\notin \bigcup_{s>0} D(e^{s|A|^{1/\beta}})
={\mathscr E}^{\{\beta\}}(A)
\]
which, by Proposition \ref{particular}, further implies that the weak solution $y(t)=e^{tA}f$, $t\ge 0$, 
of equation \eqref{1} does not belong to the 
Gevrey class ${\mathscr E}^{\{\beta \}}\left( [0,+\infty),X\right)$ of Roumieu
type and completes our consideration of the case of
the sequence's $\{\Rep\lambda_n\}_{n=1}^\infty$ being \textit{unbounded above}. 

With every possibility concerning $\{\Rep\lambda_n\}_{n=1}^\infty$ considered, 
the proof by contrapositive of the implication 
$\text{(ii)}\Rightarrow \text{(iii)}$ is complete and so is the proof of the 
entire statement.
\end{proof}

For $\beta=1$, we obtain the following important particular case.

\begin{cor}[Characterization of the Entireness of Weak Solutions]\label{CEWS}\ \\
Let $A$ be a scalar type spectral operator in a complex Banach space $(X,\|\cdot\|)$. Every weak solution of equation \eqref{1} is an entire vector function iff there is a $b_+>0$ such that the set $\sigma(A)\setminus {\mathscr P}^{1}_{b_+}$, where
\begin{equation*}
{\mathscr P}^{1}_{b_+}:= \left\{ \lambda \in \C\,\middle|\,
\Rep\lambda \ge b_+|\Imp\lambda|\right\},
\end{equation*}
is bounded (see Fig. \ref{fig:graph5}).

\begin{figure}[h]
\centering
\includegraphics[height=1.6in]{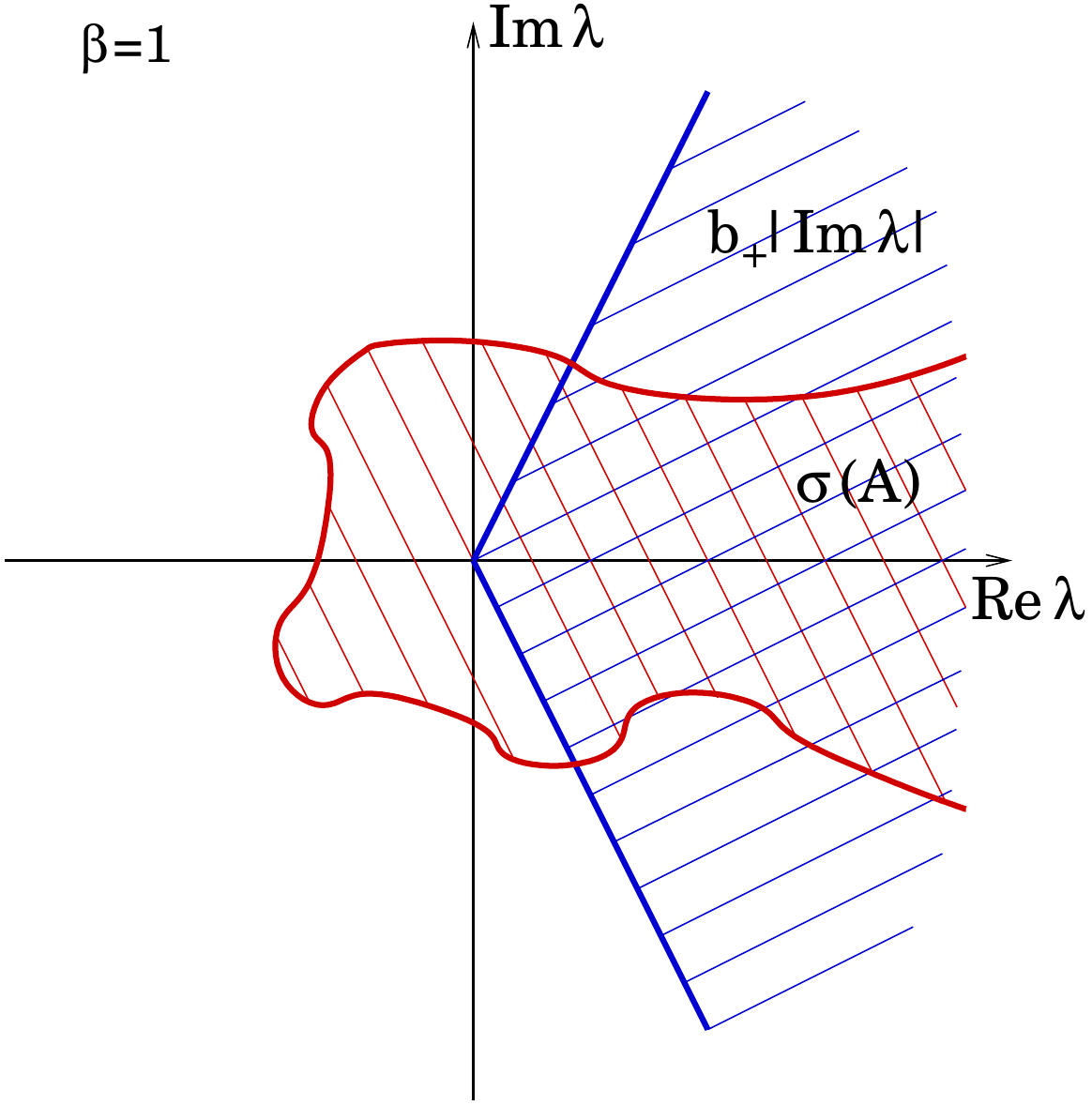}
\caption[]{}
\label{fig:graph5}
\end{figure}
\end{cor}

Observe that the region ${\mathscr P}^{1}_{b_+}$
is an angular sector with the vertex at the origin,
bisected by the positive $x$-semi-axis (see Fig. \ref{fig:graph5}).

Thus, we have obtained generalizations of Theorem $4.1$ and the entireness characterization contained in \cite{Markin2001(1)}, the counterparts for a normal operator $A$ in a complex Hilbert space.

As follows from the prior characterization, all weak solutions of equation \eqref{1} with a scalar type spectral operator $A$ in a complex Banach space can 
attain the level of strong smoothness as high as \textit{entireness} while the operator $A$ remains \textit{unbounded}, e.g., when $A$ is a semibounded below \textit{self-adjoint} unbounded operator in a complex Hilbert space (see {\cite[Corollary $4.1$]{Markin2001(1)}} and, for symmetric operators, {\cite[Theorem $6.1$]{Markin2001(1)}}). This fact contrasts the situation when a closed densely defined linear operator $A$ in a complex Banach space generates a $C_0$-semigroup, in which case the strong differentiability of all weak solutions of \eqref{1} at $0$ immediately implies \textit{boundedness} for $A$ (cf. \cite{Engel-Nagel}, see also \cite{Markin2002(2)}).

\section{Certain Inherent Smoothness Improvement Effects}

Theorem \ref{closed} implies, in particular, that
\begin{quote}
\textit{if, for some $ 1\le \beta<\infty$, every weak solution of equation \eqref{1} with a scalar type spectral operator $A$ in a complex Banach space $X$ belongs to the Gevrey class 
${\mathscr E}^{\{\beta\}}\left([0,\infty),X\right)$
of Roumieu type, then all of them belong to the narrower Gevrey class
${\mathscr E}^{(\beta)}\left([0,\infty),X\right)$
of Beurling type},
\end{quote} 
which is a jump-like effect of the weak solutions' smoothness improvement.

Notably, for $\beta =1$, we obtain the following statement: 
\begin{quote}
\textit{if every weak solution of equation \eqref{1} with a scalar type spectral operator $A$ in a complex Banach space $X$ is analytically continuable into a complex neighborhood $[0,\infty)$ (each one into its own), then all of them are entire vector functions},
\end{quote}
which can be further strengthened as follows.

\begin{prop}\label{smimp}
Let $A$ be a scalar type spectral operator in a complex Banach space $(X,\|\cdot\|)$. If every weak solution of equation \eqref{1} is analytically continuable into a complex neighborhood of $0$ (each one into its own), then all of them are entire vector functions.
\end{prop}

\begin{proof}\quad
Let us show first that, if a weak solution $y(\cdot)$ of 
equation \eqref{1} is analytically continuable into a
complex neighborhood of $0$, then $y(0)$ is an {\it analytic vector} of the operator $A$, i.e.,
\begin{equation*}
y(0)\in {\mathscr E}^{\{1\}}(A).
\end{equation*}

Let a weak solution $y(\cdot)$ of equation \eqref{1}
be analytically continuable
into a complex neighborhood of $0$. This implies that there is a $\delta>0$ such that
\begin{equation*}
y(t)=\sum_{n=0}^\infty \dfrac{y^{(n)}(0)}{n!}t^n,\ t\in [0,\delta].
\end{equation*}

The power series converging at $t=\delta$, there is a $c>0$ such that
\begin{equation*}
\biggl\|\dfrac{y^{(n)}(0)}{n!}\delta^n\biggr\|\le c,\ n\in Z_+.
\end{equation*}

Whence, considering that, by Proposition \ref{Cor} with $I=[0,\delta]$,
\[
y(0)\in C^\infty(A)\ \text{and}\ y^{(n)}(0)=A^n y(0),\
n\in\Z_+,
\]
we infer that
\begin{equation*}
\|A^ny(0)\|=\|y^{(n)}(0)\|\le c\left[\delta^{-1}\right]^n n!,\ n\in\Z_+,
\end{equation*}
which implies
\begin{equation*}
y(0)\in {\mathscr E}^{\{1\}}(A).
\end{equation*}

Now, let us prove the statement {\it by contrapositive}
assuming that there is a weak solution of 
equation \eqref{1}, which is not an entire vector function. This, by Theorem \ref{closed} with $\beta=1$, implies that there is a weak solution $y(\cdot)$ of 
equation \eqref{1}, which is not analytically continuable into a complex neighborhood of $[0,\infty)$. Then, by Proposition \ref{particular}, for some
$t_0\ge 0$,
\[
y(t_0)\not \in {\mathscr E}^{\{1\}}(A).
\]

Therefore, for the weak solution
\[
y_{t_0}(t):=y(t+t_0),\ t\ge 0,
\]
of equation \eqref{1},
\[
y_{t_0}(0)=y(t_0)\notin {\mathscr E}^{\{1\}}(A),
\]
which, as is shown above, implies that $y_{t_0}(\cdot)$
is not analytically continuable into a complex neighborhood of $0$, and hence, completes the proof by contrapositive.
\end{proof}

Thus, we have obtained a generalization of {\cite[Proposition $5.1$]{Markin2001(1)}}, the counterpart for a normal operator $A$ in a complex Hilbert space.

\section{Concluding Remark}

Due to the {\it scalar type spectrality} of the operator $A$, Theorem \ref{closed} is stated exclusively in terms of the location of its {\it spectrum} in the complex plane as well as the celebrated \textit{Lyapunov stability theorem} \cite{Lyapunov1892} 
(cf. {\cite[Ch. I, Theorem 2.10]{Engel-Nagel}}),
and thus, is intrinsically qualitative (cf. \cite{Markin2011}).

\section{Dedication and Acknowledgments}

This work is gratefully dedicated to Dr. Valentina I. Gorbachuk, a remarkable person and mathematician, whose life and work have profoundly inspired the author. 

The author also extends sincere appreciation to his colleague, Dr.~Maria Nogin of the Department of Mathematics, California State University, Fresno, for her kind assistance with the graphics.



\begin{thebibliography}{99}
\bibitem{Ball}
{J.M. Ball},	
\textit{Strongly continuous semigroups, weak solutions, and the variation of constants formula},	
{Proc. Amer. Math. Soc.}	
\textbf{63}	
{(1977)},	
{no.~2},	
{101--107}.	
\bibitem{Survey58}
{N. Dunford},	
\textit{A survey of the theory of spectral operators}, 
{Bull. Amer. Math. Soc.}	
\textbf{64}	
{(1958)},	
{217--274}.	
\bibitem{Dun-SchI}
{N. Dunford and J.T. Schwartz with the assistance of W.G. Bade and R.G. Bartle},	
\textit{Linear Operators. Part I: General Theory},	
{Interscience Publishers},	
{New York},		
{1958}.		
\bibitem{Dun-SchII}
{\bysame},	
\textit{Linear Operators. Part II: Spectral Theory. Self Adjoint Operators in Hilbert Space}, 
{Interscience Publishers},	
{New York},		
{1963}.		
\bibitem{Dun-SchIII}
{\bysame},	
\textit{Linear Operators. Part III: Spectral Operators}, 
{Interscience Publishers},	
{New York},		
{1971}.		
\bibitem{Engel-Nagel}
{K.-J. Engel and R. Nagel},	
\textit{One-Parameter Semigroups
for Linear Evolution Equations}, 
{Graduate Texts in Mathematics, vol. 194},	
{Springer-Verlag},	
{New York},		
{2000}.		
\bibitem{Gevrey}
{M. Gevrey},	
\textit{Sur la nature analytique des solutions 
des  \'equations aux d\'eriv\'ees partielles},	
{Ann. \'Ec. Norm. Sup. Paris}	
\textbf{35}	
{(1918)},	
{129--196}.	
\bibitem{Goodman}
{R. Goodman},	
\textit{Analytic and entire vectors for representations of Lie groups},	
{Trans. Amer. Math. Soc.}	
\textbf{143}	
{(1969)},	
{55--76}.	
\bibitem{GorV83}
{V.I. Gorbachuk},	
\textit{Spaces of infinitely differentiable vectors of a nonnegative self-adjoint operator},	
{Ukrainian Math. J.}	
\textbf{35}	
{(1983)},	
{531--534}.	
\bibitem{book}
{V.I. Gorbachuk and M.L. Gorbachuk},	
\textit{Boundary Value Problems for Operator Differential Equations},	
{Mathematics and Its Applications (Soviet Series), vol.~48},	
{Kluwer Academic Publishers Group},	
{Dordrecht},		
{1991}.		
\bibitem{Gor-Knyaz}
{V.I. Gorbachuk and A.V. Knyazyuk},	
\textit{Boundary values of solutions of op\-erator-differential equations},	
{Russ. Math. Surveys}	
\textbf{44}	
{(1989)},	
{67--111}.	
\bibitem{Halmos}
{P.R. Halmos},	
\textit{Measure Theory}, 
{Graduate Texts in Mathematics, vol. 18},	
{Springer-Verlag},	
{New York},		
{1974}.		
\bibitem{Hille-Phillips}
{E. Hille and R.S. Phillips},	
\textit{Functional Analysis and Semi-groups},	
{American Mathematical Society Colloquium Publications, vol.~31},	
{American Mathematical Society},	
{Providence, RI},		
{1957}.		
\bibitem{Komatsu1}
{H. Komatsu},	
\textit{Ultradistributions and hyperfunctions. Hyperfunctions and pseudo-differential equations},	
{Lecture Notes in Math.},	
{vol.~287},	
{164--179},	
{Springer},	
{Berlin},		
{1973}.		
\bibitem{Komatsu2}
{\bysame},	
\textit{Ultradistributions. I. Structure theorems and a characterization},	
{J. Fac. Sci. Univ. Tokyo Sect. IA Math.}	
\textbf{20}	
{(1973)},	
{25--105}.	
\bibitem{Komatsu3}
{\bysame},	
\textit{Microlocal analysis in Gevrey classes and in complex domains},	
{Lecture Notes in Math.},	
{vol.~1495},	
{161--236},	
{Springer},	
{Berlin},		
{1991}.		
\bibitem{Levin}
{B.Ja. Levin},	
\textit{Distribution of Zeros of Entire Functions}, 
{Translations of Mathematical Monographs, vol. 5},	
{American Mathematical Society},	
{Providence, RI},		
{1980}.		
\bibitem{Lyapunov1892}
{A.M. Lyapunov},	
\textit{Stability of Motion},	
{Ph.D. Thesis},	
{Kharkov},	
{1892},		
{English Translation},	
{Academic Press},	
{New York-London},		
{1966}.	
\bibitem{Mandel}
{S. Mandelbrojt},	
\textit{S\'eries de Fourier et Classes Quasi-Analytiques de Fonctions},	
{Gauthier-Villars},	
{Paris},		
{1935}.		
\bibitem{Markin1999}
{M.V. Markin},	
\textit{On the strong smoothness of weak solutions of an abstract evolution equation. I. Differentiability},	
{Appl. Anal.}	
\textbf{73}	
{(1999)},	
{no.~3-4},	
{573--606}.	
\bibitem{Markin2001(1)}
{\bysame},	
\textit{On the strong smoothness of weak solutions of an abstract evolution equation. II. Gevrey ultra\-differentiability}, 
{Ibid.}	
\textbf{78}	
{(2001)},	
{no.~1-2},	
{97--137}.	
\bibitem{Markin2001(2)}
{\bysame},	
\textit{On the strong smoothness of weak solutions of an abstract evolution equation. III. Gevrey ultradifferentiability of orders less than one}, 
{Ibid.}	
\textbf{78}	
{(2001)},	
{no.~1-2},	
{139--152}.	
\bibitem{Markin2002(1)}
{\bysame},	
\textit{On an abstract evolution equation with a spectral operator of scalar type},	
{Int. J. Math. Math. Sci.}	
\textbf{32}	
{(2002)},	
{no.~9},	
{555--563}.	
\bibitem{Markin2002(2)}
{\bysame},	
\textit{A note on the spectral operators of scalar type and semigroups of bounded linear operators},	
{Ibid.}	
\textbf{32}	
{(2002)},	
{no.~10},	
{635--640}.	
\bibitem{Markin2004(1)}
{\bysame},	
\textit{On scalar type spectral operators, infinite differentiable and Gevrey ultradifferentiable $C_0$-semigroups},	
{Ibid.}	
\textbf{2004}	
{(2004)},	
{no.~45},	
{2401--2422}.	
\bibitem{Markin2004(2)}
{\bysame},	
\textit{On the Carleman classes of vectors of a scalar type spectral operator},	
{Ibid.}	
\textbf{2004}	
{(2004)},	
{no.~60},	
{3219--3235}.	
\bibitem{Markin2011}
{\bysame},	
\textit{On the differentiability of weak solutions of an abstract evolution equation with a scalar type spectral operator},	
{Ibid.}	
\textbf{2011}	
{(2011)},	
{Article ID 825951},	
{27 pp.}	
\bibitem{Markin2015}
{\bysame},	
\textit{On the Carleman ultradifferentiable vectors of a scalar type spectral operator},	
{Methods Funct. Anal. Topology} 
\textbf{21}	
{(2015)},	
{no.~4},	
{361--369}.	
\bibitem{Markin2017(2)}
{\bysame},	
\textit{On the mean ergodicity of weak solutions of an abstract evolution equation},	
{Ibid.}	
{(to appear)}.	
\bibitem{Nelson}
{E. Nelson},	
\textit{Analytic vectors},	
{Ann. of Math. (2)}	
\textbf{70}	
{(1959)},	
{no.~3},	
{572--615}.	
\bibitem{Plesner}
{A.I. Plesner},	
\textit{Spectral Theory of Linear Operators},	
{Nauka},	
{Moscow},		
{1965}		
{(Russian)}.
\bibitem{Radyno1983(1)}
{Ya.V. Radyno},	
\textit{The space of vectors of exponential type},	
{Dokl. Akad. Nauk BSSR}	
\textbf{27}	
{(1983)},	
{no.~9},	
{791–-793}	
{(Russian with English summary)}.
\bibitem{Wermer}
{J. Wermer},	
\textit{Commuting spectral measures on Hilbert space},	
{Pacific J. Math.}	
\textbf{4}	
{(1954)},	
{no.~3},	
{355--361}.	
\end{thebibliography}
\end{document}